\documentclass[smallextended,referee,envcountsect]{svjour3}       % onecolumn (second format)

\smartqed  % flush right qed marks, e.g. at end of proof
\usepackage{graphicx}
\usepackage{amsmath,amssymb}
\usepackage{subcaption}
\usepackage[ruled,vlined]{algorithm2e}
\usepackage{hyperref}
\usepackage{xcolor}
\definecolor{eqcolor}{RGB}{0,0,180}        % dark blue
\definecolor{lemcolor}{RGB}{0,128,0}       % dark green
\definecolor{thmcolor}{RGB}{180,0,0}       % dark red
\definecolor{defcolor}{RGB}{128,0,128}     % purple

\usepackage{cleveref}
\hypersetup{
  colorlinks=true,
  linkcolor=black,   % section refs
  citecolor=blue,    % <--- citations color
  urlcolor=blue      % urls
}

\crefname{equation}{Eq.}{Eqs.}
\crefname{theorem}{Theorem}{Theorems}
\crefname{lemma}{Lemma}{Lemmas}
\crefname{definition}{Definition}{Definitions}
\crefname{prop}{Proposition}{Propositions}

\newcommand{\defref}[1]{\textcolor{defcolor}{Definition~\ref{#1}}}

\usepackage[a4paper, margin=1in]{geometry}  % ← adjust margin if needed
\usepackage[numbers]{natbib} % or [authoryear] based on the journal style
\usepackage{etoolbox}
\numberwithin{equation}{section}

\makeatletter
\def\makeheadbox{}
\makeatother

\journalname{}  % Will be filled by the editor

\title{Nash Equilibrium of Bi-objective Optimal Control of Fractional Space-Time Parabolic PDE}

\author{Kedarnath Buda \and
        B.V. Rathish Kumar \and
        Anil Rathi
}

\institute{Kedarnath Buda \at
              Department of Mathematics and Statistics,\\
              Indian Institute of Technology Kanpur, Kanpur 208016, India\\
              \email{kedarnath.buda@gmail.com}
           \and
           B.V. Rathish Kumar (Corresponding Author) \at
              Department of Mathematics and Statistics,\\
              Indian Institute of Technology Kanpur, Kanpur 208016, India\\
              \email{bvrk.paper.1707@gmail.com}
           \and
           Anil Rathi \at
              Department of Mathematics and Statistics,\\
              Indian Institute of Technology Kanpur, Kanpur 208016, India\\
              \email{anil.rathi.1707@gmail.com}
}

\begin{document}

\maketitle

\begin{abstract}
% We establish the existence and uniqueness of the Nash equilibrium for a bi-objective optimal control problem governed by a fractional parabolic equation involving the fractional Laplacian of order $s \in (0,1)$ and the Caputo fractional derivative of order $\gamma \in (0,1)$. The problem is modeled as a distributed control problem with quadratic cost functionals. Numerical results validate the theoretical estimates and demonstrate the effectiveness of the proposed computational scheme.
This work investigates the existence and uniqueness of the Nash equilibrium (solutions to competitive problems in which individual controls aim at separate desired states) for a bi-objective optimal control problem governed by a fractional space–time parabolic partial differential equation. The governing equation involves a Caputo fractional derivative with respect to time of order $\gamma \in (0,1)$ and a fractional Laplacian in the spatial variables of order $s \in (0,1)$. The system is associated with two independent controls, aiming at different targets. The problem is formulated as a distributed optimal control with quadratic cost functionals. Existence and uniqueness of the Nash equilibrium are established under convexity and coercivity assumptions. The solution is computed using conjugate gradient algorithms applied iteratively to the discretized optimal control problems. The numerical experiments are consistent with the theoretical estimates and illustrate the efficiency of the proposed scheme.


\end{abstract}
% \subclass{49J53 \and 49K99}

\section{Introduction}
Over the years, optimal control theory has established itself as an important part of applied mathematics, offering methods to manage and analyze dynamical systems. Its impact is particularly profound when applied to systems governed by partial differential equations (PDEs), where it provides powerful tools for tackling real-world challenges in areas such as fluid mechanics, engineering design, financial modeling, and image reconstruction~\cite{kirk2004optimal, troltzsch2010optimal, athans2007optimal}. Foundational contributions by J.L. Lions~\cite{lions1971optimal, lions1994some} laid the groundwork for the theoretical development of optimal control in PDEs.

% More recently, researchers have focused on extending optimal control techniques to fractional PDEs, motivated by their applicability to phenomena exhibiting anomalous diffusion or memory effects. These models are particularly relevant in areas such as physics, biology, and engineering~\cite{antil2016space, warma2022existence}. Fractional PDEs involve derivatives of non-integer order either in time or space, which leads to nonlocal behavior and introduces substantial analytical and numerical challenges.
In recent years, growing attention has been directed toward extending optimal control methods to systems governed by fractional partial differential equations (PDEs). The motivation arises from their remarkable ability to capture complex dynamics such as anomalous diffusion and memory effects, which cannot be described adequately by classical models. Such fractional models have proven highly relevant across diverse fields, including physics, biology, and engineering~\cite{antil2016space, warma2022existence}. Unlike standard PDEs, fractional PDEs incorporate derivatives of non-integer order in time or space, introducing nonlocality into the system. This nonlocal behavior enriches the modeling capacity but simultaneously presents significant analytical and computational challenges.

% In spatial domains, the fractional Laplacian is a key operator used to describe nonlocal diffusion processes. It is especially useful when modeling L{\'e}vy flight-type random processes, which more accurately capture the behavior of many physical systems compared to classical Brownian motion~\cite{del2003front, metzler2000random, ros2014dirichlet}. Several studies have explored the mathematical and computational aspects of the fractional Laplacian, revealing rich functional-analytic structures and regularity properties~\cite{di2012hitchhikers, cozzi2017interior, biccari2017local}.
In spatial domains, the fractional Laplacian plays a central role in describing nonlocal diffusion phenomena. Unlike the classical Laplacian, it effectively models jump-driven dynamics such as L{\'e}vy flight-type random processes, which provide a more accurate representation of many real-world processes than Brownian motion~\cite{del2003front, metzler2000random, ros2014dirichlet}. Over the past decade, several studies have explored the mathematical and computational aspects of the fractional Laplacian, revealing rich functional-analytic structures and regularity properties~\cite{di2012hitchhikers, cozzi2017interior, biccari2017local}.

% On the temporal side, the Caputo fractional derivative is widely employed to model memory effects in viscoelastic materials and porous media~\cite{allen2016parabolic, samko1993fractional, diethelm2002analysis}. The use of the Caputo derivative has been rigorously justified in several studies~\cite{chen2017time, hairer2018fractional}, highlighting its appropriateness for real-world modeling.
On the temporal side, the Caputo fractional derivative, a powerful tool for modeling processes in viscoelastic materials and porous media~\cite{allen2016parabolic, samko1993fractional, diethelm2002analysis}. The mathematical soundness of the Caputo derivative has been firmly established in numerous studies~\cite{chen2017time, hairer2018fractional}. The Caputo derivative has been mathematically justified in various studies~\cite{chen2017time, hairer2018fractional}, confirming its suitability for fractional-order modeling.

% Within the realm of fractional PDEs, control problems have received increasing attention. A notable example is the work of Micu and Zuazua~\cite{MicuZuazua2006}, which addressed the null controllability of fractional parabolic equations, offering foundational insights into the nature of controllability in such nonlocal systems.

When multiple decision-makers (or control functions) aim to optimize their respective objectives within a shared system, the mathematically  extends into the  game theory. In this setting, the concept of a \emph{Nash equilibrium} plays a pivotal role, representing a state in which no player can improve their outcome by unilaterally changing their strategy. Such configurations are particularly relevant in distributed control, economics, and multi-agent systems~\cite{ramos2002nash, borzi2013formulation, ramos2003numerical}.

In differential game theory, each player is associated with an objective functional that reflects their control goal. The equilibrium configuration reflects the optimal strategy for each player under the constraint that the others do not deviate. This multi-objective control perspective allows for a richer formulation of real-world problems where competing interests must be balanced simultaneously.

\bigskip

\noindent
\textbf{Novelty of This Work.} 
% While Nash equilibrium concepts for bi-objective optimal control of classical parabolic PDEs have been thoroughly studied in the literature~\cite{borzi2013formulation, ramos2002nash, ramos2003numerical, ramos2023nash}, all these works pertain to the classical integer-order PDE setting. The idea of Nash equilibrium is extended in this paper to bi-objective optimal control problems that are governed by fractional-in-time and fractional-in-space PDEs. In particular, we examine PDEs that involve the fractional Laplacian in space and the Caputo derivative in time.
The concept of Nash equilibrium in bi-objective optimal control has been extensively studied in the context of classical elliptic and parabolic PDEs~\cite{borzi2013formulation, ramos2002nash, ramos2003numerical, ramos2023nash}. However, these studies are limited to the traditional integer-order setting. In this work, we extend the concept of Nash equilibrium to bi-objective optimal control problems governed by fractional PDEs, where both time and
space derivatives are of non-integer order. Specifically, our focus is on systems involving the Caputo fractional derivative in time, which captures memory effects, and the fractional Laplacian in space, which accounts for nonlocal diffusion.

% We begin by showing that the Nash equilibrium for these fractional PDE control problems exists and is unique. We then suggest a numerical method to calculate these Nash solutions and use numerical experiments to demonstrate the results.  This setting introduces new mathematical challenges due to the nonlocality in time and space, extending the classical theory of differential games into the fractional framework. Here, we address the fractional Laplacian in space and the nonlocality of Caputo time.

% By adding two different cost functionals, $J_1$ and $J_2$, connected to independent control agents, bi-objective control builds on conventional control theory.  The equilibrium we seek is characterized by the simultaneous minimization of these functionals under PDE constraints.

We first establish the existence and uniqueness of Nash equilibrium for the proposed fractional PDE control problems. To approximate these equilibrium, we design a numerical scheme and validate its effectiveness through computational experiments. The fractional setting introduces significant mathematical challenges, primarily due to the nonlocal nature of both space and time operators. In particular, our analysis addresses the interplay between the spatial fractional Laplacian and the temporal nonlocality induced by the Caputo derivative, thereby extending the classical theory of differential games into the fractional domain. Bi-objective control extends classical control theory by introducing two distinct cost functionals, $J_1$ and $J_2$, each associated with an independent control agent. The Nash equilibrium in this setting corresponds to a pair of strategies that simultaneously minimize these objectives under the governing PDE constraints.

\paragraph{Structure of the Paper.}
The paper is organized as follows. In Section~\ref{sec:2}, we present the analytical and terminological prerequisites. We begin with a general overview of the fractional Laplacian and then discuss key concepts related to fractional derivatives and integral, which play a crucial role in variational analysis. In Section~\ref{sec:3}, we introduce the bi-objective optimal control problem governed by a time-fractional parabolic partial differential equation, incorporating the Caputo derivative in time and the fractional Laplacian in space. Section~\ref{sec:4} is devoted to the theoretical analysis of the problem, where we establish and characterize the Nash equilibrium and derive the associated optimality system, concluding with the existence and uniqueness of the Nash equilibrium for the main problem. In Section~\ref{sec:5}, we illustrate the practical performance of our theoretical results through several numerical examples, all solved using a conjugate gradient algorithm. Example~\ref{ex:1}, implemented in FreeFEM++, investigates a time-fractional parabolic problem, while Example~\ref{ex:2}, carried out in MATLAB, addresses an elliptic optimal control problem involving the fractional Laplacian. Furthermore, we analyze the influence of different fractional orders on the solutions. Finally, Examples~\ref{ex:3} and \ref{ex4} examine one-dimensional time-fractional parabolic Nash equilibrium problems. These simulations are implemented in Python, employing the L1 scheme \cite{jin2016analysis} for the Caputo time derivative and the fractional Laplacian discretization proposed in \cite{duo2018novel}.

\section{Preliminaries: Analytical Setting}\label{sec:2}
Let's start by providing background information on the fractional Laplacian and the related functional structure that is important to nonlocal PDEs and their optimal control in order to prepare the reader for our next analysis.

\subsection*{Fractional Laplacian and Functional Setting:}
To proceed, we recall fundamental concepts related to the fractional Laplacian, which plays a crucial role in the analysis and optimal control of nonlocal partial differential equations.

Let \( s \in (0,1) \). The fractional Laplacian \( (-\Delta)^s \) of a function \( w : \mathbb{R}^N \to \mathbb{R} \) is defined by the singular integral
\begin{equation}
(-\Delta)^s w(x) := C_{N,s} \, \text{P.V.} \int_{\mathbb{R}^N} \frac{w(x) - w(y)}{|x - y|^{N + 2s}} \, dy, \label{eq:fractional_Laplacian}
\end{equation}
where \text{P.V.} denotes the Cauchy principal value, and the constant \( C_{N,s} \) is given by
\[
C_{N,s} := \frac{s \cdot 2^{2s} \Gamma\left( \frac{N + 2s}{2} \right)}{\pi^{N/2} \Gamma(1 - s)}.
\]

For functions \( w \in \mathcal{C}^{1,1}_{\mathrm{loc}}(\mathbb{R}^N) \cap \mathcal{L}^1_s(\mathbb{R}^N) \), the singular integral defining the fractional Laplacian is well-defined in the principal value sense:
\[
(-\Delta)^s w(x) := C_{N,s} \lim_{\epsilon \to 0} \int_{\mathbb{R}^N \setminus B_\epsilon(x)} \frac{w(x) - w(y)}{|x - y|^{N + 2s}} \, dy,
\]
where the weight space \( \mathcal{L}_s^1(\mathbb{R}^N) \) consists of locally integrable functions with sufficient decay at infinity:
\[
\mathcal{L}_s^1(\mathbb{R}^N) := \left\{ w \in L^1_{\mathrm{loc}}(\mathbb{R}^N) : \int_{\mathbb{R}^N} \frac{|w(x)|}{1 + |x|^{N + 2s}} \, dx < \infty \right\}.
\]
This setting ensures both the regularity near the singularity \( x = y \) and the integrability at infinity, thereby making the principal value integral convergent. Note that when \( s \in \left(0, \tfrac{1}{2}\right) \), the integrand becomes integrable near \( x \), and the singularity in the principal value vanishes.

The operator \( (-\Delta)^s \) is inherently nonlocal. Unlike the classical Laplacian, its evaluation at a point \( x \) depends on the values of the function on the entire domain. For instance, if \( w > 0 \) in a bounded open set \( \Omega \subset \mathbb{R}^N \) and \( w \equiv 0 \) in \( \mathbb{R}^N \setminus \Omega \), then for any \( x \in \mathbb{R}^N \setminus \Omega \),
\[
(-\Delta)^s w(x) = C_{N,s} \int_{\Omega} \frac{-w(y)}{|x - y|^{N + 2s}} \, dy < 0,
\]
highlighting the global influence of the function values over the entire space.

\medskip

The fractional Sobolev space is defined by
\[
H^s(\mathbb{R}^N) := \left\{ w \in L^2(\mathbb{R}^N) : [w]_{H^s(\mathbb{R}^N)} < \infty \right\},
\]
with the Gagliardo seminorm
\[
[w]_{H^s(\mathbb{R}^N)}^2 := \int_{\mathbb{R}^N} \int_{\mathbb{R}^N} \frac{|w(x) - w(y)|^2}{|x - y|^{N + 2s}} \, dx \, dy.
\]
The space \( H_0^s(\Omega) \subset H^s(\mathbb{R}^N) \) consists of functions vanishing almost everywhere outside \( \Omega \), i.e.,
\[
H_0^s(\Omega) := \left\{ w \in H^s(\mathbb{R}^N) : w = 0 \text{ a.e. in } \mathbb{R}^N \setminus \Omega \right\},
\]
and its dual is denoted by \( H^{-s}(\Omega) \) (see \cite{di2012hitchhikers}). 
% For further details, see \cite{di2012hitchhikers}.

\subsection*{Fractional Derivatives, Integrals, and Integration by Parts:}
Let \( \gamma \in (0,1) \). The left-sided Caputo derivative of order \( \gamma \) for a function \( w : \Omega \times (0,T) \to \mathbb{R} \) is defined as
\[
\partial_t^\gamma w(x,t) := \frac{1}{\Gamma(1 - \gamma)} \int_0^t \frac{1}{(t - \xi)^\gamma} \frac{\partial w(x,\xi)}{\partial \xi} \, d\xi.
\]
The corresponding right-sided Caputo derivative is
\[
\partial_{T - t}^\gamma w(x,t) := \frac{-1}{\Gamma(1 - \gamma)} \int_t^T \frac{1}{(\xi - t)^\gamma} \frac{\partial w(x,\xi)}{\partial \xi} \, d\xi.
\]
We also define the Riemann–Liouville fractional integrals of order \( \gamma > 0 \) as:
\[
(I_t^\gamma w)(t) := \frac{1}{\Gamma(\gamma)} \int_0^t \frac{w(\xi)}{(t - \xi)^{1 - \gamma}} \, d\xi, \qquad
(I_{T - t}^\gamma w)(t) := \frac{1}{\Gamma(\gamma)} \int_t^T \frac{w(\xi)}{(\xi - t)^{1 - \gamma}} \, d\xi.
\]

Let \( \chi \) be a Banach space (e.g., \( L^2(\Omega), H^s(\Omega), H^{-s}(\Omega) \)). Then the Bochner space \( L^p(0,T;\chi) \) for \( 1 \leq p \leq \infty \) is defined by
\[
\|w\|_{L^p(0,T;\chi)} = \left( \int_0^T \|w(t)\|_\chi^p dt \right)^{1/p}, \quad
\|w\|_{L^\infty(0,T;\chi)} = \operatorname{ess\,sup}_{t \in (0,T)} \|w(t)\|_\chi.
\]
Next, we now introduce the solution space used in the variational formulation, defined as in \cite{antil2016space}:
\[
\mathbb{V} := \left\{ w \in L^\infty(0,T;L^2(\Omega)) \cap L^2(0,T;H^s_0(\Omega)) : \partial_t^\gamma w \in L^2(0,T;H^{-s}(\Omega)) \right\}.
\]
This construction allows us to invoke the following fractional integration by parts identity, which is essential in deriving the adjoint equations.

\begin{lemma}[Fractional Integration by Parts, {\cite[Lemma 7]{antil2016space}}] \label{lemma:frac_ibp}
Let \( v, w \in \mathbb{V} \cap C([0,T]; L^2(\Omega)) \). If Let $\gamma \in (0,1)$. Then the following identity holds:
\[
\int_0^T \langle \partial_t^\gamma v(t), w(t) \rangle - \langle \partial_{T-t}^\gamma w(t), v(t) \rangle \, dt
= (w(T), (I_t^{1-\gamma} v)(T)) - (v(0), (I_{T-t}^{1-\gamma} w)(0)),
\]
where \( (\cdot, \cdot) \) denotes the \( L^2(\Omega) \) inner product, and \( \langle \cdot, \cdot \rangle \) is the duality pairing between \( H^{-s}(\Omega) \) and \( H^s(\Omega) \).
\end{lemma}

\section{Problem Formulation: Bi-objective Optimal Control of a Space-Time Fractional Parabolic PDE}\label{sec:3}
This section presents the formulation of a bi-objective optimal control problem constrained by a space-time fractional parabolic partial differential equation (PDE), with its Nash equilibrium. This formulation extends the classical PDE-constrained bi-objective optimal control problem by incorporating Caputo-type time-fractional and fractional Laplacian operators.

Our aim is to find a control pair \( (u_1, u_2) \in U_1 \times U_2 \subset L^2(\Omega \times (0, T))^2 \) such that the following two-player Nash optimality system holds:
\begin{equation}
\begin{aligned}
&\textbf{Player 1:} \\
&\quad \min_{u_1 \in U_1} J_1(w, u_1, u_2), \quad \text{subject to:} \\
&\quad \partial_t^\gamma w + (-\Delta)^s w = B_1 u_1 + B_2 u_2 + f \quad \text{in } \Omega \times (0,T), \\
&\quad w = 0 \quad \text{on } \Omega^c\times (0,T), \\
&\quad w(0) = w_0 \quad \text{in } \Omega, \\
\\
&\textbf{Player 2:} \\
&\quad \min_{u_2 \in U_2} J_2(w, u_1, u_2), \quad \text{subject to:} \\
&\quad \partial_t^\gamma w + (-\Delta)^s w = B_1 u_1 + B_2 u_2 + f \quad \text{in } \Omega \times (0,T), \\
&\quad w = 0 \quad \text{on } \Omega^c\times (0,T), \\
&\quad w(0) = w_0 \quad \text{in } \Omega.
\end{aligned}
\label{eq:bi-objective_control_equation}
\end{equation}
Here, \( \Omega \subset \mathbb{R}^N \) is a bounded Lipschitz domain, \( w_0 \in L^2(\Omega) \) is the initial state, and \( f \in L^2(\Omega \times (0,T)) \) is a given source term. \( \partial_t^\gamma \) is the fractional time derivative of order \( \gamma \in (0,1) \), \( (-\Delta)^s \) is the fractional Laplacian of order \( s > 0 \), \( B_1, B_2 \) are linear operators, and \( u_1 \in U_1 \), \( u_2 \in U_2 \) are controls, with \( U_1, U_2 \) being admissible control sets, \( Q = \Omega \times (0,T) \), with \( \Omega \subset \mathbb{R}^n \) and \( T > 0 \), \( \mathbb{V} \) is a suitable function space for \( w \) and \( p_1 \).
\medskip

The cost functional for each player \( j \in \{1, 2\} \) is defined by
\[
J_j(w, u_1, u_2) := \frac{1}{2} \int_0^T \left( \|w - w_j\|^2_{L^2(\Omega)} + \mu_j \|B_j u_j\|^2_{L^2(\Omega)} \right) dt,
\]
where \( w_j \in L^2(\Omega \times (0,T)) \) is a desired state and \( \mu_j > 0 \) is a regularization parameter. The admissible control sets are defined by
\[
U_j := \left\{ u \in L^2(\omega_j \times (0,T)) : l_j(x,t) \leq u(x,t) \leq r_j(x,t) \text{ a.e. in } \omega_j \times (0,T) \right\},
\]
where \( \omega_j \subset \Omega \) is an open subdomain, and \( l_j, r_j \in L^\infty(\omega_j \times (0,T)) \) with \( l_j(x,t) \leq r_j(x,t) \) a.e. for \( j = 1, 2 \).
\medskip

Further, we introduce the \emph{control extension operator} 
\[
B_j : L^2(\omega_j \times (0,T)) \;\longrightarrow\; L^2(\Omega \times (0,T)), \quad j=1,2,
\]
which is defined by
\begin{equation}\label{eq:control-extension}
B_j u_j(x,t) :=
\begin{cases}
u_j(x,t), & (x,t) \in \omega_j \times (0,T), \\[6pt]
0,        & (x,t) \in (\Omega \setminus \omega_j) \times (0,T).
\end{cases}
\end{equation}
In this way, each control function \(u_j\) is naturally extended from its local support 
\(\omega_j \subset \Omega\) to the entire space--time domain \(\Omega \times (0,T)\). 
Consequently, the influence of the control remains localized to the prescribed region, 
while the formulation of the PDE dynamics is preserved on the whole domain.

\medskip

The above formulation defines a non-cooperative bi-objective optimization problem in which each control \( u_j \) aims to steer the common state \( w \) toward a distinct target \( w_j \), under the influence of their counterpart's control. The concept of Nash equilibrium is applied to determine a pair of admissible controls \( (u_1, u_2) \) such that neither player can unilaterally reduce their own cost.

\noindent \textbf{Notation:} From now onwards, define the reduced cost functionals 
$$J_j(w(u_1,u_2),u_1,u_2): = J_j(u_1,u_2).$$

% In this section, we show the Nash equilibrium to the bi-objective control problem \eqref{eq:bi-objective_control_equation}.  exists and is unique. 

% We have $B_i, i=1,2$ be the linear extension operator as   $B_i u_i\in L^2(\Omega \times (0,T))$.

\begin{theorem}[Existence and uniqueness] The optimal control problem 
\begin{equation}\label{maincost}
    J_j(u_1,u_2) := \frac{1}{2} \int_{0}^{T} \Big( \|w-w_j\|^2_{L^2(\Omega)} + \mu_j \|B_j u_j\|^2_{L^2(\Omega)} \Big) dt
\end{equation}

subjected to the PDE
\begin{equation}
\begin{cases} 
\partial ^ \gamma _t w + (-\Delta)^sw=B_1u_1+B_2u_2+ f  \text{ in } \Omega \times (0,T),  \\
w=0 \text{ on } \Omega^c \times (0,T), \\
w(0)=w_0  \text{ in } \Omega, \\
u_j \in U_j,
\end{cases}
\label{eq:mainstate}
\end{equation}
has a unique optimal solution $\bar u_j(u_i) \in U_i$ (i.e., $J_j(\bar u_j(u_i),u_i) \leq J_j(u_j,u_i) , ~ \forall u_j \in U_i$) for each $u_j \in U_j(j\neq i) $ and $f,w_i \in L^2(\Omega \times (0,T) ), \text{ for } i,j \in 1,2$.
\end{theorem}
\begin{proof}
Fix $u_i \in U_i$ with $i \neq j$. Then the control-to-state operator $S_j$ is affine, continuous, 
and maps $L^2(\Omega \times (0,T))$ into itself. Moreover, the reduced cost functional $J_j$ $(j=1,2)$ 
is strictly convex, weakly lower semi-continuous, and coercive. Therefore, by the direct method of 
the calculus of variations, $J_j$ admits a unique minimizer in the admissible set $U_j$; 
see Section~2, Theorem~10 of~\cite{antil2016space} for a detailed proof.
\end{proof}

\section{Existence and uniqueness of a Nash equilibrium.} \label{sec:4}
In this section, we show that the existence and uniqueness of  Nash equilibrium to the problem \eqref{eq:bi-objective_control_equation}.

The above problem is formulated within the settings of a noncooperative Nash game, in which each player \( j \in \{1,2\} \) seeks to minimize their own cost functional \( J_j( u_1, u_2) \), subject to the common state dynamics governed by a space-time fractional parabolic PDE. Crucially, the control variable \( u_j \) for player \( j \) influences the global state \( w \), which in turn affects the cost functional of the other player.

In this setting, neither player can improve their outcome by unilaterally deviating from their strategy. The equilibrium thus represents a stable operating point of the coupled control system, consistent with the classical game-theoretic notion of Nash equilibrium applied to distributed parameter systems.

\begin{definition}
A feasible element $\bar u=(\bar u_1,\bar u_2)\in U_{ad}$ is called a 
\emph{Nash solution} or \emph{Nash equilibrium} of 
\eqref{eq:bi-objective_control_equation} if 
\begin{equation}
\begin{cases}
J_1(\bar u_1,\bar u_2) \leq J_1(u_1,\bar u_2) \quad \forall u_1 \in U_1, \\
J_2(\bar u_1,\bar u_2) \leq J_2(\bar u_1, u_2) \quad \forall u_2 \in  U_2.
\end{cases}   \label{eq:nash_solution_definition}
\end{equation}
\end{definition}
Thus, a Nash strategy consists of two controls acting competitively: 
each player seeks to minimize their own functional while the opponent’s strategy is fixed.

Consider the map 

\[
\begin{aligned}
J_1 : &~ L^2\!\big((0,T)\times \Omega\big) \;\longrightarrow\; \mathbb{R}, 
&\quad \text{defined by } u_1 &\;\longmapsto\; J_1(u_1, \bar u_2), \\[6pt] 
J_2 : &~ L^2\!\big((0,T)\times \Omega\big) \;\longrightarrow\; \mathbb{R}, 
&\quad \text{defined by } u_2 &\;\longmapsto\; J_2(\bar u_1, u_2).
\end{aligned}
\]
both $J_{1}$ and $J_{2}$ are convex function for fixed $\bar u_2$ and $\bar u_1$, respectively. 

It is straightforward to verify that, for fixed $\bar u_2$ (resp.\ $\bar u_1$), the reduced mappings are convex. This allows us to characterize the minimizers via variational inequalities. Let's  begin with a classical result.

\begin{lemma}[Variational Inequality]\label{variq}
Let \( J_j \) be the cost functional defined in \eqref{maincost}.  
Then, for fixed \( \bar u_i \) with \( i \neq j \), a control  
\( \bar{u}_j \in U_j \) minimizes \( J_j \) if and only if  
\[
   \left( \frac{D J_j(\bar{u}_1, \bar{u}_2)}{D u_j},\,
   u_j - \bar{u}_j \right)_{L^2(\Omega \times (0,T))} \;\geq\; 0,
   \quad \forall u_j \in U_j .
\]

\end{lemma}

\begin{proof}
The concept is standard; (see \cite[Lemma~2.21]{troltzsch2010optimal} and \cite[Lemma~14]{antil2016space}).
\end{proof}

\noindent The condition stated in \eqref{eq:nash_solution_definition} can be characterized by the following result.
\begin{lemma}\label{l3}
Let \( J_j \) be the cost functional defined in \eqref{maincost} associated with player \(j=1,2\),  
and let \( U_j \) denote the corresponding admissible control set.  
A pair of admissible controls \( (\bar{u}_1, \bar{u}_2) \in U_1 \times U_2 \) 
constitutes a Nash equilibrium if and only if the following first-order optimality conditions hold:
\begin{equation}\label{eq:nash_eq}
\begin{cases}
\displaystyle 
\frac{D J_1(\bar u_1,\bar u_2)}{D u_1} \cdot (u_1-\bar u_1) \;\geq\; 0, 
& \forall\, u_1 \in U_1, \\[1.2ex]
\displaystyle 
\frac{D J_2(\bar u_1,\bar u_2)}{D u_2} \cdot (u_2-\bar u_2) \;\geq\; 0, 
& \forall\, u_2 \in U_2.
\end{cases}
\end{equation}
\end{lemma}

\begin{proof}
By combining the definition of a Nash equilibrium (see \defref{eq:nash_solution_definition}) with Lemma \ref{variq} on variational inequalities, we conclude that a pair \((\bar u_1,\bar u_2) \in U_1 \times U_2\) is a Nash equilibrium if and only if it satisfies the coupled system of variational inequalities given in \eqref{eq:nash_eq}.
\end{proof}

\subsection{Formal Lagrangian Formulation for the Modified Optimal Control Problem:}
Here, we derive the first-order necessary and sufficient optimality conditions for the optimal control problem defined by equations \eqref{eq:bi-objective_control_equation} with a modified Lagrangian, following the approach described in \cite[Section 3.1]{antil2016space}. Although the computations are formal, they provide insight into the correct form of the optimality conditions through a simple and intuitive procedure.

\noindent Letting \( p_j \) denote the adjoint variable, the Lagrangian \( \mathcal{L}_j : \mathbb{V} \times U_j \times \mathbb{V} \to \mathbb{R} \) is defined as:
\[
\mathcal{L}_j(w, u_j, p_j) = J_j( u_1,u_2) - \int_Q \left( \partial_t^\gamma w + (-\Delta)^s w - f - B_1 u_1 + B_2 u_2 \right) p_j \, dx \, dt. \label{eq:Lagrangian}
\]
The necessary and sufficient optimality conditions at the optimal triplet \( (\bar{w}, \bar{u}_j, \bar{p}_j) \) \cite{ antil2016space,troltzsch2010optimal} are:
\begin{align}
\frac{D\mathcal{L}_j(\bar{w}, \bar{u}_j, \bar{p}_1)}{Dp_j}. h &= 0, & \forall h \in \mathbb{V}, \label{eq:D_pL} \\
\frac{D\mathcal{L}_1(\bar{w}, \bar{u}_1, \bar{p}_1)}{Dw} .h &= 0, & \forall h \in \mathbb{V}, \text{ with } h(0) = 0, \label{eq:D_wL} \\
\frac{D\mathcal{L}_j(\bar{w}, \bar{u}_j, \bar{p}_j)}{Du_j}.(u_j - \bar{u}_j) &\geq 0, & \forall u_j \in U_j. \label{eq:D_uL}
\end{align}

The expression $\frac{D\mathcal{L}_j(\bar{w}, \bar{u}_j, \bar{p}_j)}{Dp_j}\,\cdot h$ denotes the \textit{directional derivative of \(\mathcal{L}_j\) with respect to \(p_j\)} at the point \((\bar{w}, \bar{u}_j, \bar{p}_j)\) in the direction of \(h\). Similarly, directional derivatives with respect to \(\,w\) and \(\,u_j\) can be defined in the same way.

\subsection{Integration by Parts:}
We start with the constraint term in the Lagrangian:
\[
\int_{L^2(\Omega \times (0,T))} \left( \partial_t^\gamma \bar{w} + (-\Delta)^s \bar{w} - f - B_1 \bar{u}_1 + B_2 u_2 \right) \bar{p}_j \, dx \, dt. \label{eq:constraint_term}
\]
Using the integration-by-parts formula for fractional derivatives (see Lemma~\ref{lemma:frac_ibp}), we obtain  
\[
\int_{L^2(\Omega \times (0,T))} \left( \partial_t^\gamma \bar{w} \right) \bar{p}_j \, dx \, dt = \int_{L^2(\Omega \times (0,T))} \left( \partial_{T-t}^\gamma \bar{p}_j \right) \bar{w} \, dx \, dt - \int_\Omega \bar{w}(0) (I_{T-t}^{1-\gamma} \bar{p}_j)(0) \, dx + \int_\Omega \bar{p}_j(T) (I_t^{1-\gamma} \bar{w})(T) \, dx, \label{eq:ibp_time}
\]
where the fractional integrals are:
\begin{align*}
(I_t^{1-\gamma} \bar w)(T) &= \frac{1}{\Gamma(1-\gamma)} \int_0^T (t-s)^{-\gamma} \bar w(s) \, ds, \\
(I_{T-t}^{1-\gamma} \bar p_j)(t) &= \frac{1}{\Gamma(1-\gamma)} \int_t^T (s-t)^{-\gamma} \bar p_j(s) \, ds.
\end{align*}
For the fractional Laplacian, assuming that \( (-\Delta)^s \) is self-adjoint, we have:
\[ 
\int_Q \left( (-\Delta)^s \bar{w} \right) \bar{p}_j \, dx \, dt = \int_Q \bar{w} \left( (-\Delta)^s \bar{p}_j \right) \, dx \, dt. \label{eq:ibp_space}
\]
Thus, the constraint term \eqref{eq:constraint_term} becomes:
\begin{align*}
&\int_{L^2(\Omega \times (0,T))} \left( \partial_t^\gamma \bar{w} + (-\Delta)^s \bar{w} - f - B_1 \bar{u}_1 + B_2 u_2 \right) \bar{p}_j \, dx \, dt \notag \\
&= \int_{L^2(\Omega \times (0,T))} \left( \partial_{T-t}^\gamma \bar{p}_j + (-\Delta)^s \bar{p}_j \right) \bar{w} \, dx \, dt - \int_{L^2(\Omega \times (0,T))} (f + B_1 \bar{u}_1 - B_2 u_2) \bar{p}_j \, dx \, dt \notag \\
&\quad - \int_\Omega \bar{w}(0) (I_{T-t}^{1-\gamma} \bar{p}_j)(0) \, dx + \int_\Omega \bar{p}_j(T) (I_t^{1-\gamma} \bar{w})(T) \, dx. \label{eq:constraint_rewritten}
\end{align*}

\subsubsection{Condition \eqref{eq:D_wL}: Derivative with Respect to \( w \):}
Based on the previous computation, we can rewrite \eqref{eq:D_uL} in the form
\begin{equation}\label{eq:3.6}
\frac{D\mathcal{L}_j(\bar w,\bar u_j,\bar p_j)}{Dw}.h 
= (\bar w-w_j,h)_{L^2(Q)} 
- (\partial_t^\gamma \bar p_j + (-\Delta)^s \bar p_j, h)_{L^2(\Omega \times (0,T))}
- (\bar p_j(T), (I_t^{1-\gamma}h)(T))_{L^2(\Omega)} = 0,
\end{equation}
for all \(h \in W\) such that \(h(0)=0\).

Now, let \(\varphi \in C_0^\infty(0,T)\) and \(\psi \in C_0^\infty(\Omega)\) be arbitrary. Define
\[
h(x,t) = \psi(x)\,\phi(t),
\]
where \(\phi\) solves the Abel integral equation
\[
(I_t^{1-\gamma}\phi)(t) = \varphi(t).
\]
The unique solvability of this Abel equation is guaranteed by \cite[Section~2.2, Theorem~2.1]{samko1993fractional}, with this choice of \(\phi\), we observe that
\[
(I_t^{1-\gamma}h)(T)=0.
\]

To derive the adjoint equation, choose \( h = \psi \phi \), where \( \psi \in C_0^\infty(\Omega) \), and \( \phi \in C_0^\infty(0,T) \) solves the Abel equation \( (I_t^{1-\gamma} \phi)(t) = \varphi(t) \), with \( \varphi \in C_0^\infty(0,T) \). By \cite[Theorem 2.1]{samko1993fractional}, this is solvable, and \( (I_t^{1-\gamma} h)(T) = \psi (I_t^{1-\gamma} \phi)(T) = \psi \varphi(T) = 0 \). Substituting into \eqref{eq:3.6}:
\[
\int_0^T \phi(t) \left\langle \partial_{T-t}^\gamma \bar{p}_j + (-\Delta)^s \bar{p}_j - (\bar{w} - w_j), \psi \right\rangle_{L^2(\Omega)} \, dt = 0. 
\]
Since the range of \( I_t^{1-\gamma} \) includes all smooth, compactly supported functions (see \cite[Theorems 13.2, 13.5]{samko1993fractional}), and \( \psi \in C_0^\infty(\Omega) \) is arbitrary, we obtain \eqref{eq:adjoint_eq}. Now we will see the terminal condition. 
\smallskip

\noindent\textbf{Derivation of the adjoint equation.}  
To proceed, let $\varphi \in C_0^\infty(0,T)$ and $\psi \in C_0^\infty(\Omega)$ be arbitrary.  
Define the test function
\[
h(x,t) = \psi(x)\phi(t),
\]
where $\phi$ solves the Abel integral equation
\[
(I_t^{1-\gamma}\phi)(t) = \varphi(t), \qquad \varphi \in C_0^\infty(0,T).
\]
The unique solvability of this Abel equation is guaranteed by \cite[Section~2.2, Theorem~2.1]{samko1993fractional}.  
With this choice, we observe that $(I_t^{1-\gamma}h)(T)=0$.  

Substituting into the equation \eqref{eq:3.6} gives
\[\label{eq:test_form}
\int_0^T \varphi(t)\, \big(\partial_{T-t}^\gamma \bar p_j + (-\Delta)^s \bar p_j - (\bar w-w_j), \psi\big)_{L^2(\Omega)} \,dt = 0.
\]
Since the range of $I_t^{1-\gamma}$ contains all smooth, compactly supported functions (see \cite[Theorems 13.2, 13.5]{samko1993fractional}), and $\psi \in C_0^\infty(\Omega)$ is arbitrary, we conclude that
\begin{equation}\label{eq:adjoint_eq}
\partial_{T-t}^\gamma \bar p_j + (-\Delta)^s \bar p_j = \bar w - w_j \quad \text{in } \Omega\times (0,T).
\end{equation}
For the terminal condition, consider the term:
\[
\left( \bar{p}_j(T), (I_t^{1-\gamma} h)(T) \right)_{L^2(\Omega)} = 0. \label{eq:terminal_term}
\]
Set \( h = \ell_\epsilon(t) \chi \), with \( \chi \in C_0^\infty(\Omega) \), and:

\[
\ell_\epsilon(t) = \begin{cases} 
\epsilon^{-\gamma} T^{-\gamma} t^\gamma, & 0 < t \leq \epsilon T, \\
1, & \epsilon T < t \leq T.
\end{cases} 
\]
Then:
\begin{equation}
\left( \bar{p}_j(T), \, \big(I_t^{1-\gamma} \ell_\epsilon(t)\chi\big)(T) \right)_{L^2(\Omega)}
=  \, \big(I_t^{1-\gamma} \ell_\epsilon\big)(T)\,
\left( \bar{p}_j(T), \chi \right)_{L^2(\Omega)}= 0.
\label{eq:terminal_condition}
\end{equation}

As \( \epsilon \to 0 \), \( (I_t^{1-\gamma} \ell_\epsilon)(T) \to (I_t^{1-\gamma} 1)(T) = \frac{T^{1-\gamma}}{\Gamma(2-\gamma)} > 0 \), so \( \bar{p}_j(T) = 0 \).
Collecting the derived equations, our formal argument yields the following strong system for the adjoint variable $p_j$.
Let us introduce the adjoint state $p_j$ to get a convenient expression.

\begin{definition}[Fractional Adjoint State]
Given controls \( u_1, u_2 \in L^2((0,T); H^{-s}(\Omega)) \), let 
\( w = w(B_1 u_1 + B_2 u_2) \) denote the solution to the state system. 
Then the adjoint states \( p_j \in \mathbb{V} \), \( j=1,2 \), are defined as the unique weak solutions to the following backward-in-time space-time fractional equations:
\[
\begin{cases}
\partial_{T-t}^\gamma p_j + (-\Delta)^s p_j = w - w_j & \text{in } \Omega \times (0,T), \\
p_j = 0 & \text{on } \Omega^c\times (0,T), \\
p_j(T) = 0 & \text{in } \Omega,
\end{cases}
\quad j=1,2.
\label{eq:fractional_adjoint_state}
\]

The functions \( p_1(u_1, u_2) \) and \( p_2(u_1, u_2) \) are referred to as the \textit{fractional adjoint states} corresponding to the control pair \( (u_1, u_2) \). The well-posedness of such backward-in-time fractional equations can be established using standard theory for linear fractional parabolic problems (see \cite{jin2016analysis}).
\end{definition}

\subsubsection{Condition \eqref{eq:D_pL}: Derivative with Respect to \( p_j \):}
The Fr{\'e}chet derivative of \( \mathcal{L}_j \) with respect to \( p_j \) in the direction \( h \in \mathbb{V} \) is:
\begin{align*}
\frac{D\mathcal{L}_j(\bar{w}, \bar{u}_j, \bar{p}_1)}{Dp_j} .h &= \lim_{\epsilon \to 0} \frac{\mathcal{L}_j(\bar{w}, \bar{u}_j, \bar{p}_j + \epsilon h) - \mathcal{L}_j(\bar{w}, \bar{u}_j, \bar{p}_j)}{\epsilon} \notag \\
&= - \int_Q \left( \partial_t^\gamma \bar{w} + (-\Delta)^s \bar{w} - f - B_1 \bar{u}_1 + B_2 u_2 \right) h \, dx \, dt. \label{eq:D_pL_deriv}
\end{align*}
Setting \( \frac{D\mathcal{L}_j}{Dp_j}(\bar{w}, \bar{u}_j, \bar{p}_j).h = 0 \) for all \( h \in \mathbb{V} \), we obtain the \textit{state equation}:
\eqref{eq:mainstate}

\subsubsection{Condition \eqref{eq:D_uL}: Derivative with Respect to \( u_j \):}
The Fr{\'e}chet derivative with respect to \( u_j \) in the direction \( u_j - \bar{u}_j \in U_j \) is:
\begin{align*}
\frac{D\mathcal{L}_j(\bar{w}, \bar{u}_j, \bar{p}_j)}{Du_j}.(u_j - \bar{u}_j) &= \frac{D J_j(\bar{w}, \bar{u}_j)}{Du_j}.(u_j - \bar{u}_j) + \int_{\Omega \times (0,T)} (B_j (u_j - \bar{u}_j)) \bar{p}_j \, dx \, dt \notag \\
&= \frac{D J_j(\bar{w}, \bar{u}_j)}{Du_j}.(u_j - \bar{u}_j) + (u_j - \bar{u}_j, B_j^* \bar{p}_j)_{L^2(\Omega \times (0,T)}, 
\end{align*}
where \( B_j^* \) is the adjoint operator of \( B_j \). 

Assuming \( J_j(w, u_j) = \frac{1}{2} \| w - w_j \|_{L^2(L^2(\Omega \times (0,T)}^2 + \frac{\mu_j}{2} \| B_j u_j \|_{L^2(L^2(\Omega \times (0,T)}^2 \), we have:
\[
\frac{D J_j(\bar{u}_j, \bar{u}_2)}{Du_j}.(u_j - \bar{u}_j) = \mu_j (B_j \bar{u}_j, B_j (u_j - \bar{u}_j))_{L^2(\Omega \times (0,T))} = \mu_j (B_j^* B_j \bar{u}_j, u_j - \bar{u}_j)_{L^2(\Omega \times (0,T))}.
\]
Thus, the the necessary and suﬃcient first-order variational inequality is:
\begin{equation}
    (\mu_j B_j^* B_j \bar{u}_j + B_j^* \bar{p}_j, u_j - \bar{u}_j)_{L^2(L^2(\Omega \times (0,T)} \geq 0 \quad \forall u_j \in U_j. \label{eq:variational_ineq}
\end{equation}

\paragraph{Remarks.}  
\begin{itemize}
  \item The operator \(B_j\), together with its adjoint \(B_j^*\), is always to be understood  in Hilbert spaces. 
  This ensures that all inner products are well-defined and the functional setting remains rigorous.  
  \item The structure of the argument is inherently symmetric with respect to the control variables. Hence, the same line of reasoning can be applied when optimizing with respect to \(u_i\), while treating \(u_j\) as fixed for $j \neq i$
  \item The quantity  
  \[
    \frac{Dw(\bar u_1,\bar u_2)}{Du_j}\cdot B_j(u_j-\bar u_j)
  \]
  should be viewed as the \emph{linearized effect of a perturbation in the control} on the state. 
  This captures the appropriate form of the directional derivative that is crucial in the analysis of Nash equilibrium.
\end{itemize}

The operator $B_j$ introduced in \eqref{eq:control-extension} is linear and bounded. To identify its adjoint
\[
\begin{split}
\bigl( B_j u_j, p_j \bigr)_{L^2(\Omega \times (0,T))} 
&= \int_{\Omega \times (0,T)} (B_j u_j)(x,t)\, p_j(x,t)\, dxdt \\[4pt]
&= \int_{\omega_j \times (0,T)} (B_j u_j)(x,t)\, p_j(x,t)\, dxdt 
   + \int_{(\Omega \setminus \omega_j) \times (0,T)} (B_j u_j)(x,t)\, p_j(x,t)\, dxdt \\[4pt]
&= \int_{\omega_j \times (0,T)} u_j(x,t)\, p_j(x,t)\, dxdt \,+\, 0 \\[4pt]
&= \bigl( u_j, \chi_{\omega_j \times (0,T)} p_j \bigr)_{L^2(\omega_j \times (0,T))}.
\end{split}
\]
Here, $\chi_{\omega_j \times (0,T)}$ denotes the characteristic function of the set 
$\omega_j \times (0,T)$, which equals $1$ on $\omega_j \times (0,T)$ and vanishes elsewhere. Therefore, the adjoint operator is defined as follows:
\[
B_j^* : L^2(\Omega \times (0,T)) \;\longrightarrow\; L^2(\omega_j \times (0,T)),
\]
which can be written explicitly as
\[\label{eq:adjoint-operator}
B_j^* p_j = \chi_{\omega_j \times (0,T)}\, p_j, 
\qquad \forall \, p_j \in L^2(\Omega \times (0,T)).
\]
Hence, the action of $B_j^*$ is nothing but the restriction of a function $p_j$ 
defined on the entire space–time cylinder $\Omega \times (0,T)$ 
to the local control region $\omega_j \times (0,T)$, realized by multiplication 
with the characteristic function $\chi_{\omega_j \times (0,T)}$.

\noindent\textbf{Optimality System.}  
Altogether, it follows that the solution to the main problem 
\eqref{eq:bi-objective_control_equation} is characterized by the simultaneous solution of the state 
$\bar w(\bar u_1, \bar u_2)$ at $(\bar u_1, \bar u_2)$ of the 
\textbf{state equation}~\eqref{eq:mainstate}, the \textbf{adjoint equations}~\eqref{eq:fractional_adjoint_state}, 
together with the \textbf{variational inequalities}~\eqref{eq:variational_ineq}, which are satisfied for both $j=1,2$.

\begin{equation} \label{os}
\left\{
\begin{aligned} 
&\partial_t^\gamma \bar w + (-\Delta)^s \bar w = B_1 \bar u_1 + B_2 \bar u_2 + f 
&& \text{in } \Omega \times (0,T), \quad 
\bar w=0 \text{ on } \Omega^c\times (0,T), \quad  
\bar w(0)=w_0 \text{ in }\Omega, \\[6pt]
&\partial_{T-t}^\gamma \bar p_j + (-\Delta)^s \bar p_j = \bar w - w_j 
&& \text{in } \Omega \times (0,T), \quad  
\bar p_j=0 \text{ on } \Omega^c\times (0,T), \quad  
\bar p_j(T)=0 \text{ in }\Omega, \\[6pt]
&\big(\mu_j B_j^\ast B_j \bar u_j - B_j^\ast \bar p_j, \, u_j - \bar u_j \big) \geq 0 
&& \forall u_j \in U_j.
\end{aligned}
\right.
\end{equation}

We emphasize that the analysis in~\cite{borzi2013formulation, ramos2002nash} is carried out in the elliptic setting. In contrast, our problem is \emph{parabolic} and involves fractional derivatives in both space and time.  Accordingly, while we borrow parts of the variational-inequality and game-theoretic methodology from the 
elliptic case, the parabolic fractional setting necessitates a different functional setup and several 
technical modifications (e.g., time-dependent test spaces and coercivity estimates adapted to fractional 
operators). In particular, the state equation is posed in an evolution setting with a spatial fractional 
operator and a time-fractional derivative, and the resulting Nash system is formulated as a time-dependent 
variational inequality on appropriate Sobolev spaces.
\[
H := L^2(\Omega \times (0,T))^2, \qquad 
(u,v)_H := (u_1,v_1)_{L^2(\Omega \times (0,T))} + (u_2,v_2)_{L^2(\Omega \times (0,T))}, 
\quad u,v \in H,
\]
where $(\cdot,\cdot)_H$ and $\|\cdot\|$ denote the inner product and norm on $H$.

\medskip
\noindent
Equivalently, the optimality system \eqref{os} can be reinterpreted in compact form as
\[
\big( \mathcal{A}(\bar u_1,\bar u_2) - (b_1,b_2), \, (u_1-\bar u_1,u_2-\bar u_2) \big)_{H} \geq 0,
\qquad \forall (u_1,u_2) \in U_1 \times U_2, 
\]
where the operator $\mathcal A:H \to H$ is defined by 
$$\mathcal A(u_1,u_2):= (\mu_1 B_1^*B_1 u_1+B_1^*\tilde p_1,\mu_2 B_2^*B_2 u_2+B_2^*\tilde p_2 ) \in H,$$
For each $(u_1,u_2)\in H$, we first compute the auxiliary state $\tilde w$ by solving  
\begin{equation}
\begin{cases}
\partial_t^\gamma \tilde w + (-\Delta)^s \tilde w \;=\; B_1 u_1 + B_2 u_2 & \text{in } \Omega\times (0,T),  \\
\tilde w = 0 & \text{in } \Omega^c\times (0,T),  \\
\tilde w(0)=0 & \text{in } \Omega .
\end{cases}
\label{eq:equation_solving_w}
\end{equation}

In the next step, the adjoint variables $\tilde p_j$ for $j=1,2$ are obtained by solving
\begin{equation}
\begin{cases}
\partial_{T-t}^\gamma \tilde p_j + (-\Delta)^s \tilde p_j \;=\; \tilde w & \text{in } \Omega\times (0,T), \\
\tilde p_j = 0 & \text{on } \Omega^c \times (0,T), \\
\tilde p_j(T)=0 & \text{in } \Omega ,
\end{cases}
\quad j=1,2.
\label{eq:equation_solving_p1_p2}
\end{equation}

Notice that  $\tilde p_1$ and $\tilde p_2$ coincide in our case. However, they would be different, in general, considering different tracking functionals.

We have that $\mathcal A$ results in being the linear part of the optimality conditions operator, which is an affine mapping. In fact, the inhomogeneous term $b=(b_1,b_2) $ is defined in terms of $f$ and getting desired functions $w_1$ and $w_2$, and it is zero when these functions are zero. Specifically, the construction of $b_1$ and $b_2$ proceed as follows.
Define $\hat w$ as the solution of 
\[\
\begin{cases}
\partial ^ \gamma _t \hat w + (-\Delta)^s \hat w=f  \text{ in } \Omega\times (0,T), \\
\hat w=0 \text{ on } \Omega^c \times (0,T), \\
 \hat w(0)=w_0  \text{ in } \Omega. \\
\end{cases}   \label{eq:equation_solving_w_cap}
\]
Hence, $b_1$ and $b_2$ are given by $b_j=B_j^*\hat p_j$ for $j=1,2$ solves the equation 
\[
\begin{cases}
\partial ^ \gamma _{T-t} \hat p_j + (-\Delta)^s\hat p_j=\hat w-w_j  \text{ in } \Omega\times (0,T), \\
 \hat p_j=0 \text{ on } \Omega^c \times (0,T), \\
 \hat p_j(T)=0  \text{ in } \Omega.  \\
\end{cases}   \label{eq:equation_solving_pj_cap}
\]

We can express the solution of the state system~\eqref{eq:mainstate} and, for $j=1,2$, of~\eqref{eq:equation_solving_p1_p2} as  
\[
w = \tilde{w} + \hat{w}, \qquad 
p_i = \tilde{p}_j + \hat{p}_j,
\]
respectively.
 
Now, we define the mapping $a:H\times H \to \mathbb{R}$ is defined by 
$$a(u,v):= (\mathcal A u-b,v)_H \quad \forall u,v \in H.$$
Thus, $\bar u=(\bar u_1,\bar u_2)$ is a Nash equilibrium if and only if it satisfies the following variational inequality. 
\begin{equation}\label{mainvariational}
a(\bar u , u-\bar u ) \geq 0 \quad \forall u \in U_1 \times U_2.
\end{equation}

\begin{proposition}\label{proposition}
Assuming the mappings $B_j^*B_j$ are coercive. Then, the mapping $a: H\times H \to \mathbb R$ from the above equations is bilinear, continuous, and coercive. 
\end{proposition}

\begin{proof}
It is evident that the operator $\mathcal A$ is continuous, bounded, and linear in $H$. As a result, we immediately determine that $a$ is continuous and bilinear. We are going to show the coercivity of $a$. Consider 
\[
\begin{split}
(\mathcal A(u_1,u_2),(u_1,u_2))_H & = 
 \mu_1(B_1^*B_1u_1,u_1)_{L^2(\omega_1\times (0,T))}
 + \mu_2(B_2^*B_2u_2,u_2)_{L^2(\omega_2\times (0,T))}\\
& + \int_{\omega_1 \times (0,T)} (B_1^*p_1)u_1 \,\mathrm{d}x\, \mathrm{d}t + \int_{\omega_2 \times (0,T)} (B_2^*p_2)u_2 \,\mathrm{d}x\, \mathrm{d}t.
 \end{split}   \label{eq:inner_product_(Au,u)}
\] 

In view of the assumed coercivity of $B_j^*B_j$, the statement follows if we are able to show that the last two terms are non-negative. 
To this end, we first recall that we have $\tilde p_1 =\tilde  p_2 $ in our setting \eqref{eq:equation_solving_p1_p2}. Therefore, we obtain 
\begin{equation*}
\begin{split}
\int_{\omega_1 \times (0,T)} (B_1^*p_1)u_1 \,\mathrm{d}x\, \mathrm{d}t + \int_{\omega_2 \times (0,T)} (B_2^*p_2)u_2 \,\mathrm{d}x\, \mathrm{d}t 
& = \int_{\Omega\times (0,T)} \tilde p_1(B_1u_1)\,\mathrm{d}x\, \mathrm{d}t +  \int_{\Omega\times (0,T)} \tilde p_2(B_2u_2)\,\mathrm{d}x\, \mathrm{d}t \\
& =\int_{\Omega \times (0,T)} \tilde p_1(B_1u_1+ B_2u_2)\, \mathrm{d}x\,\mathrm{d}t\\
& = \int_{\Omega \times (0,T)} \tilde p_1(\partial ^ \gamma _t \tilde w + (-\Delta)^s\tilde w)\,\mathrm{d}x\,\mathrm{d}t\\
 &= \int_{\Omega \times (0,T)} (\partial ^ \gamma _{T-t} \tilde p_1 + (-\Delta)^s \tilde p_1 )\tilde w \, \mathrm{d}x\,\mathrm{d}t \\
 & = \int_{\Omega \times (0,T)} \tilde w \tilde w \,\mathrm{d}x\,\mathrm{d}t.
   \end{split}
   \end{equation*}
Hence, the last two terms are non-negative. Therefore, coercive.
\end{proof}

\begin{theorem}\label{thm:nash-equilibrium}
The reduced problem \eqref{eq:nash_solution_definition} admits a unique Nash equilibrium. 
\end{theorem}
\begin{proof}
Note that $\bar u=(\bar u_1,\bar u_2)$ the conditions of the variational inequality \eqref{mainvariational}
$$a(\bar u , u-\bar u ) \geq 0 \quad \forall v \in U_1\times U_2.$$
Then the Lions–Stampacchia theorem (see \cite{kinderlehrer2000introduction}) implies the following problem: $$a(\bar u, u-\bar u ) \geq 0 \quad \forall v \in U_1\times U_2,$$ has a unique solution because $a$ is bilinear, continuous, and coercive.
\end{proof}

\begin{algorithm*}
    \SetAlgoLined
    \SetKwInOut{Input}{Input}
    \SetKwInOut{Output}{Output}

    \textbf{STEP 1.} $(u_1^0,u_2^0)$ is given in $U_1\times U_2$\;
    
    \textbf{STEP 2.a.} $w^0$ is the solution of state equation
    
    \textbf{STEP 2.b.} For $ i=1,2,\bar p_i^0 $ is the solution of  PDE, 
		$\begin{cases} \partial ^ \gamma _{T-t} p_i^0 + (-\Delta)^s p_i^0= (w^0-w_i) \chi_{d} \text{ in } Q\\ 
		p_i^0=0 \text{ in } \Omega^c  \times(0,T)\\
		p_i^0(T)=0 \text{ in } \Omega
		\end{cases}$\;

    \textbf{STEP 2.c.}  $(g_1^0,g_2^0)=(\mu_1u^0_1+p_1^0\chi_{\omega_1},\mu_2u^0_2+p_2^0\chi_{\omega_1})\in U_1\times U_2$\;    
 
    \textbf{STEP 3.}  $(h_1^0,h_2^0)=(g_1^0,g_2^0) \in U_1\times U_2$, \;
		for $k\geq 0$, assuming that $(u_1^k,u_2^k),(g_1^k,g_2^k),(h_1^k,h_2^k)$ are known, we compute $(u_1^{k+1},u_2^{k+1}),(g_1^{k+1},g_2^{k+1}),(h_1^{k+1},h_2^{k+1})$ as given below:\;

    \textbf{STEP 4.a.}  $\bar w^k$ is the solution of the PDE, 
		$\begin{cases}   \partial ^ \gamma _t \bar w^k +(-\Delta)^s \bar w^k = h_1^k\chi_{\omega_1}+h_2^k\chi_{\omega_2}  \text{ in } \Omega\\
                                  \bar w^k =0 \text{ in }  \Omega^c \times (0,T)\\
                                   \bar w^k(0)=0  \text{ in } \Omega 
                                  \end{cases}$ \\
                                  
                                  observe $(u_1,u_2):=(h_1^k,h_2^k)$. \;     
    \textbf{STEP 4.b.} For $i=1,2, \bar p_i^k $ is the solution of  PDE,
	          
	           $\begin{cases}
	            \partial ^ \gamma _{T-t} \bar p^k+ (-\Delta)^s \bar p_i^k= \bar y^k \text{ in } Q\\
	             \bar p_i^k=0 \text{ in }  \Omega^c \times(0,T)\\
	             \bar p_i^k(T)=0 \text{ in } \Omega
	             \end{cases}$\\
		    
    	\textbf{STEP 4.d.} $\rho_k= \frac{\|(g_1^k,g_2^k)\|^2}{\int_{\omega_1\times(0,T)}\bar g_1^k h_1^k dxdt + \int_{\omega_2\times(0,T)}\bar g_2^kh_2^k dxdt}.$\;
		\textbf{STEP 5.} $(u_1^{k+1},u_2^{k+1})=(u_1^k,u_2^k)-\rho_k (h_1^k ,h_2^k).$\;
		\textbf{STEP 6.} $(g_1^{k+1},g_2^{k+1})=(g_1^k,g_2^k)-\rho_k (\bar g_1^k ,\bar g_2^k).$\;
		\textbf{STEP 7.} $\gamma_k=\frac{\|(g_1^{k+1},g_2^{k+1)}\|^2}{\|(g_1^k,g_2^k)\|^2}$.\;
		\textbf{STEP 8.} $(h_1^{k+1},h_2^{k+1})=(g_1^{k+1},g_2^{k+1})+\gamma_k (h_1^k ,h_2^k).$\;
		\textbf{STEP 9.} Do $k=k+1$,and go to  \textbf{STEP 4.}
    
    \caption{Conjugate Gradient Method}
    \label{algo}
\end{algorithm*}

\section{Numerical Experiment}\label{sec:5}

\begin{example}\label{ex:1}
Let us consider the bi-objective optimal control problem \eqref{eq:bi-objective_control_equation} with the fractional Caputo derivative and classical Laplacian operators such as 
$$
\partial ^ \gamma _t w  - \Delta w = B_1 u_1 + B_2 u_2 + f, ~~~ \text{in} \Omega
$$
The operators $B_1$ and $B_2$ are defined as
\[
B_1 u_1 = u_1|_{\omega_1}, \qquad B_2 u_2 = u_2|_{\omega_2},
\]
corresponding to the computational domains (see Figure~\ref{fig:domain})  
\[
\begin{aligned}
\Omega &= [0,1] \times [0,1], & 
\omega_d &= [0.25,0.75] \times [0.25,0.75], \\
\omega_1 &= [0,0.25] \times [0,0.25], &
\omega_2 &= [0.75,1] \times [0,0.25].
\end{aligned}
\]

The final time for the experiment is taken as $T = 1.5$ .

\begin{figure}[h!]
\centering
\includegraphics[width=0.45\textwidth]{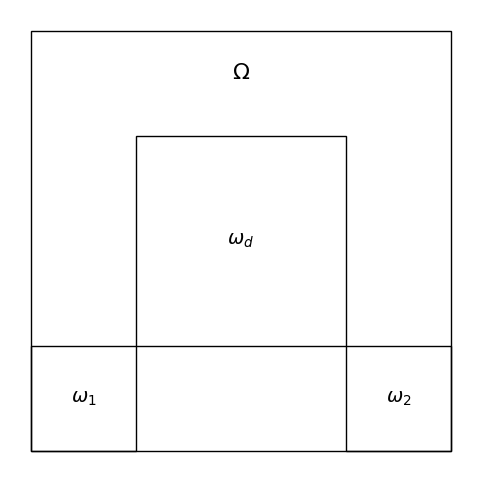}
\caption{The domain $\Omega$ and subdomains $\omega_1, \omega_2, \omega_d$.}
\label{fig:domain}
\end{figure}

\noindent To assess the Nash solution, we consider desired states $w_1 = 1$ and $w_2 = -1$ with problem data $w_0 = 0$ and $f = 1$. Using Algorithm 1 (conjugate gradient Method) with initial guess $(u_{1}^{0}, u_{2}^{0}) = (0, 0)$, the desired functions $w_1$ and $w_2$ are achieved by a Nash equilibrium strategy. The PDE is discretized using the finite element method in space and finite differences for the Caputo time derivative. In the simulations, we set the fractional order to $\gamma = 0.99$.
% We are going to see how much of the goal has been achieved by the Nash state solution to the two given desired functions $w_1$ and $w_2$. We use $w_0 = 0$ and $f = 1$ as the problem's data. We are going to use the \textbf{Algorithm 1}(Conjugate Gradient Method) for that the initial guess is chosen as $(u_{1}^{0}, u_{2}^{0}) = (0, 0)$. The desired functions $w_1 = 1$ and $w_2 = -1$ are achieved by a Nash equilibrium strategy. We use the finite element method to solve the PDE, and Caputo time discretization is done using the finite difference method. The Caputo fractional order $\gamma=0.99$ derivatives are used in the following numerical computations. 

Since $\gamma \approx 1$ means that the Caputo fractional derivative behaves similarly to the usual time derivative. Therefore, we compare the proposed algorithm with the A.M.Ramos et al. \cite{ramos2003numerical} work for the cases of $\mu_1 = \mu_2 = 0$ (No control) and $\mu_1 = \mu_2 = 10^{-4}$. In \textbf{Table 1}, we notice that the $L^2$ error between the Nash solution and the desired functions is quite similar in the domain $\omega \times (0, 1.5)$. In Figure \ref{figure:1}, we observe that, in the observable domain, the $L^2$ norm of the difference between the Nash solution and the first desired function ($w_1 = 1$) initially decreases and then after some time, it remains the same for both cases $\mu_1 = \mu_2 = 0$ and $\mu_1 = \mu_2 = 10^{-4}$ whereas the $L^2$ norm of the Nash solution and second desired function ($w_2 = -1$) in the observable domain is increasing for some time initially and then it remains constant in both cases. Figure \ref{figure:2} shows the Nash solution for the case of no control and $\mu_1 = \mu_2 = 10^{-4}$ at time  $T = 1.5$.

%%%%%%%%%%%%%%%%%%%

% Let us consider PDE problem \eqref{eq:bi-objective_control_equatio} with fractional Caputo derivative and classical Laplacian PDE, $\partial ^ \gamma _t w  -\Delta w=B_1u_1+B_2u_2+1$,  where the domain of experiment is  $\Omega=[0,1] \times [0,1], \omega_1= [0,0.25]\times [0,0.25], \omega_2 = [0.75,1]\times [0,0.25]$ with finite horizon time $T=1.5$. The observable space $w_d=[0.25,0.75] \times [0.25,0.75]$. We are going to see how much of the goal has been achieved by the Nash state solution to the two given desired functions $w_1$ and $w_2$. For the data for the problem, we take $f=1 and w_0=0$. In the conjugate gradient algorithm, we take the initial guess $(u_1^0,u_2^0)=(0,0)$. We have some different desired functions to achieve by Nash equilibrium strategy $w_1 = 1$ and $w_2 = -1$. First, we will see the FEM solution by using $FreeFEM++$ of Nash by taking different regularizing parameters. Take $\mu_1= 10^{-4}, \mu_2= 10^{-4}$. We have taken the Caputo fractional order $\gamma=0.9$ derivatives in the following numerical computations. 

%%%%%%%%%%%%%%%%%%%

The L1 scheme \cite{jin2016analysis,li2018analysis} is widely used to approximate the fractional derivative of Caputo. As an example, consider the nonlinear fractional
ordinary differential equation
\[
\frac{d^{\gamma} w(t)}{dt^{\gamma}}.
\]
For time discretization, the interval \([0, T]\) is divided uniformly with step
size \(\tau\), giving the grid points.
\[
0 = t_0 < t_1 < \dots < t_{j+1} < \dots < t_N = T, 
\quad t_j = j\tau.
\]
Based on this grid, the Caputo fractional derivative can then be discretized using
the L1 scheme.

\noindent Using the L1 scheme, the discretization of the fractional order  is presented as follows:
\[ w(t_{j+1}) = w(t_j) - a_{j-k} [w(t_{k+1}) - w(t_k)] + \tau \Gamma(2 - \gamma) , \quad k = 0 \]
Here, the coefficients \( a_{j-k} \) are defined as:
\[ a_{j-k} = (j+1-k)^{1-\gamma} - (j-k)^{1-\gamma} \]
Here, the Caputo fractional derivative is defined by
\[ C^{\alpha} 1 \int_{t_0}^{t} \frac{w(s)}{(t - s)^{\alpha'}} ds. \]
Let \( v = u - u_0 \). Then (1)-(2) can be written as the following
\[ 0 = D^{\tau} w(t_n) = \Gamma(1 - \alpha) \]
Let \(0 < t_0 < t_1 < \cdots < t_N = T\) be a partition of the time interval \([0, T]\) and \(\tau\) the timestep size. Let \(N \geq 2\). (The case for \(N = 1\) is trivial.) We first need to approximate the Caputo fractional derivative at \(t = t_n\), \(n \geq 2\). By definition, we have
\[
\frac{1}{\Gamma(1-\alpha)} \int_{0}^{t_n} (t_n - s)^{-\alpha} w'(s) ds = \frac{1}{n-1} \sum_{j=0}^{n-1} \int_{t_j}^{t_{j+1}} (t_n - s)^{-\alpha} w'(s) ds.
\]
On each subinterval \([t_j, t_{j+1}]\), we approximate \(w(s)\) by using the piecewise linear interpolation function
\[ P_1(s) = \frac{s - t_{j+1}}{t_j} w(t_j) + \frac{s - t_j}{t_{j+1}} w(t_{j+1}), \quad s \in [t_j, t_{j+1}]. \]
\[
D_t^{\alpha}w(t_n) \approx \frac{1}{n-1} \sum_{j=0}^{n-1} \int_{t_j}^{t_{j+1}} (t_n - s)^{-\alpha} w'(s) ds. (t_n - s)^{-\alpha} P_1'(s) ds,
\]
where the weights are defined by, with \(n \geq 2\),
\[
w_{j,n} = \begin{cases}
1, & \text{if } j = 0, \\
2^{1-\alpha} - 2(j-1)^{1-\alpha}, & \text{if } j = 1, \\
(j+1)^{-\alpha} - 2j^{-\alpha} + (j-1)^{-\alpha}, & \text{if } 2 \leq j \leq n-1, \\
\end{cases}
\]

\begin{table}[ht] 
    \begin{center}
        \begin{tabular}{|c|c|c|c|}
            \hline
            & Error & No Control & $\mu_1 = \mu_2 = 10^{-4}$ \\ 
            \hline
            A. M. Ramos ( with $\gamma=1$) \cite{ramos2003numerical} & $\|\bar w(t)-w_{1}\|_{L^2(\omega_{d} \times (0, 1.5))}$ & 0.330592 & 0.343811 \\
            & $\|\bar w(t)-w_{2}\|_{L^2(\omega_{d} \times (0, 1.5))}$ & 0.422275 & 0.420292 \\
            \hline
            Present Work (with $\gamma=0.99$) & $\|\bar w(t)-w_{1}\|_{L^2(\omega_{d} \times (0, 1.5))}$ & 0.330583 & 0.3405192 \\
            & $\|\bar w(t)-w_{2}\|_{L^2(\omega_{d} \times (0, 1.5))}$ & 0.422284 & 0.4292397\\
            \hline
        \end{tabular}
    \end{center}   
\caption{Norm of distance between Nash solution and desired function} \label{tab1}
\end{table}

\begin{figure}
    \centering
    \begin{subfigure}[t]{0.49\textwidth}
    \centering
    \includegraphics[width=1\textwidth]{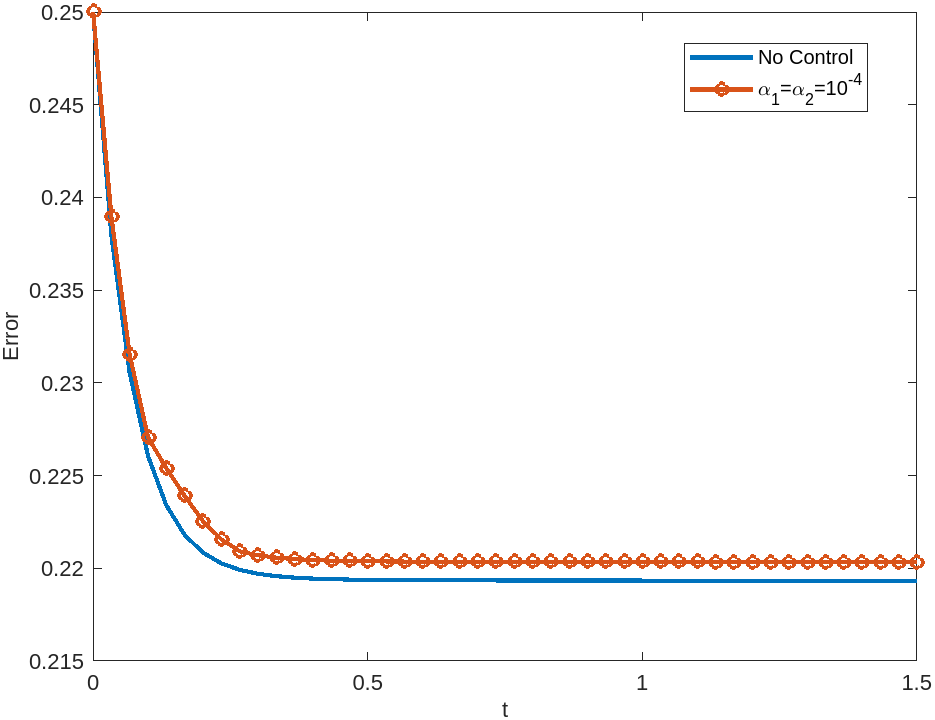}
    \caption{$\|\bar w(t)-w_{1}\|_{L^{2}(\omega_d)}^2$ }
    \label{figure:1(a)}
    \end{subfigure}
    \hfill
    \begin{subfigure}[t]{0.49\textwidth}
    \centering
    \includegraphics[width=1\textwidth]{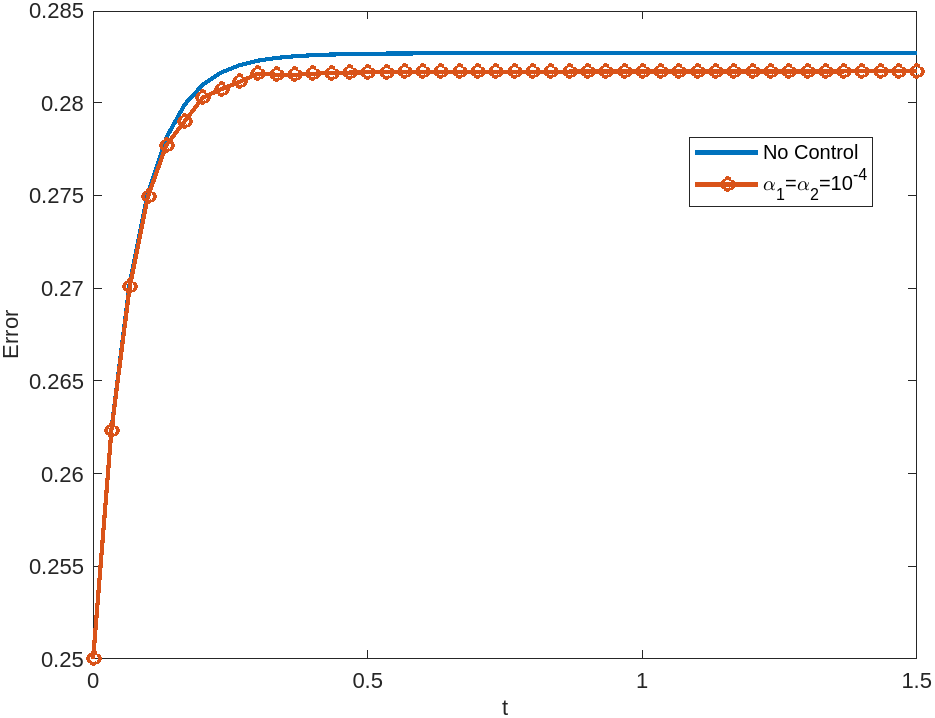}
    \caption{$\|\bar w(t)-w_{2}\|_{L^{2}(\omega_d)}^2$ }
    \label{figure:1(b)}
    \end{subfigure}
    \caption{}
    \label{figure:1}
\end{figure}

\begin{figure}
    \centering
    \begin{subfigure}[t]{0.49\textwidth}
    \centering
    \includegraphics[width=1\textwidth]{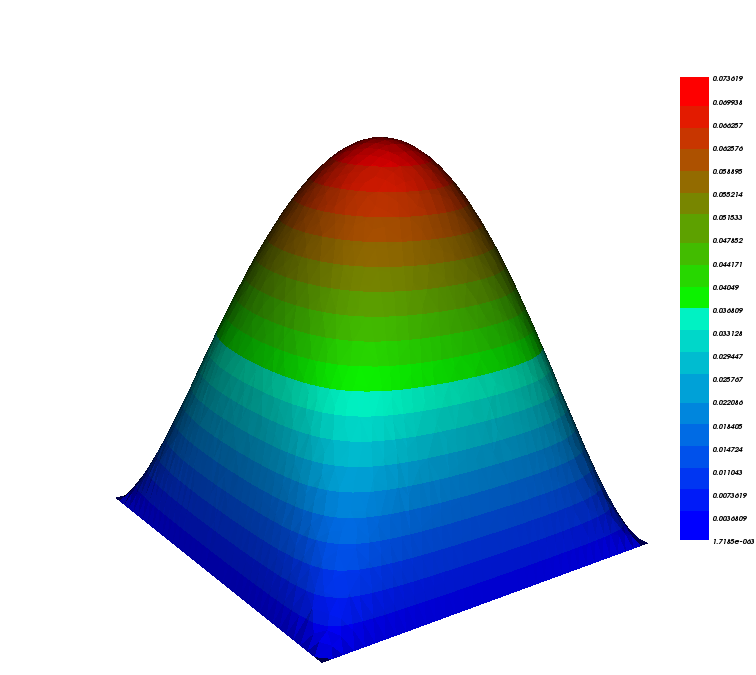}
    \caption{Non-Controlled Solution at $T=1.5$}
    \label{figure:2(a)}
    \end{subfigure}
    \hfill
    \begin{subfigure}[t]{0.49\textwidth}
    \centering
    \includegraphics[width=1\textwidth]{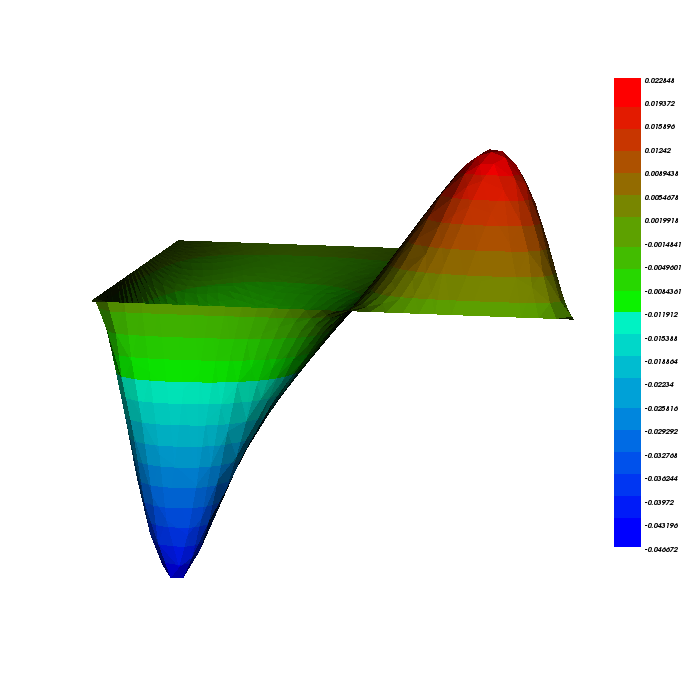}
    \caption{Computed Solution of Controlled equation at $T=1.5$ for $\mu_1 = 10^{-4}$ and $\mu_2 = 10^{-4}$ }
    \label{figure:2(b)}
    \end{subfigure}
    \caption{}
    \label{figure:2}
\end{figure}

\end{example}

\newpage
\begin{example}\label{ex:2}
Let us consider the bi-objective optimal problem \eqref{eq:bi-objective_control_equation} time-independent fractional Laplacian PDE defined for $i = 1,2$ 
$$(-\Delta)^s w = B_1 u_1 + B_2 u_2 + f, ~~~ \text{in} \Omega $$
The computational domains are considered the same as Figure~\ref{fig:domain} in Example 5.1.

The data for the problem is taken as $f = 1$ and $w_0 = 0$. We are going to see how much of the goal has been achieved by the Nash state solution to the two given desired functions $w_1 = 1$ and $w_2 = -1$. We are using the conjugate gradient algorithm with initial data as $(0, 0)$. For solving the fractional Laplacian PDE, we are going to use the finite difference method (see \cite{duo2019novel}). The Conjugate gradient method is used to solve the optimal control problem in $MATLAB$  for the Nash equilibrium state by taking different fractional orders of $s$ in the fractional Laplacian operator. Take $\mu_1= 10^{-4}, \mu_2= 10^{-4}$. In the following numerical computations, we have used the fractional Laplacian. Table~\ref{tab2} gives a description of $L^2(\Omega)$ error of the state solution and first ($w_1 = 1$) and second ($w_2 = -1$) desired function for the various values of $s$, where $s$ is the order of fractional Laplacian. Table~\ref{tab2} shows that the $L^2$ errors of the state solution and desired functions are getting smaller when $s$ tends to $1$.

%%%%%%%%%%%%%%%%%%%%%%%%%%%

% Let us consider PDE problem \eqref{eq:bi-objective_control_equatio} with time independent and fractional Laplacian PDE, $(-\Delta)^s w_i=B_1u_1+B_2u_2+1$, for $i=1,2$, where the domain of experiment is  $\Omega = [0,1] \times [0,1], \omega_1= [0,0.25]\times [0,0.25], \omega_2 = [0.75,1] \times [0,0.25]$. The observable space $w_d=[0.25,0.75] \times [0.25,0.75]$. We are going to see how much of the goal has been achieved by the Nash state solution to the two given desired functions $w_1$ and $w_2$. For the data for the problem, we take $f=1, w_0=0$. In the conjugate gradient algorithm, we take the initial guess $(u_1^0,u_2^0)=(0,0)$. We have some different desired functions to achieve by Nash equilibrium strategy $w_1= 1$ and $w_2=-1$. First, we will see the FDM for computing \cite{duo2019novel} the two-dimensional fractional Laplacian. Then, the Conjugate gradient method is used to solve the optimal control problem in $MATLAB$  for  Nash equilibrium state by taking different fractional orders of $s$ in the fractional Laplacian operator. Take $\mu_1= 10^{-4}, \mu_2= 10^{-4}$. In the following numerical computations, we have used fractional Laplacian.

\begin{table}[ht] 
\begin{center}
\begin{tabular}{|c|c|c|c|}
\hline $\mathbf{2s}$  & $ \mathbf{\|\bar w-w_{1}\|_{L^2(\Omega)}}$  & $\mathbf{ \|\bar w-w_{2}\|_{L^2(\Omega)}}$  & \textbf{Tolerance} \\
\hline 1.95 & 3.6457e-09    & 3.6457e-09    & 8.1465e-10\\
\hline 1.9	& 8.1643e-09    & 8.1643e-09	& 2.6434e-09\\
\hline 1.85	& 6.6071e-09	& 6.6071e-09	& 7.6575e-10\\
\hline 1.8	& 3.3371e-08	& 3.3371e-08	& 4.8492e-08\\
\hline 1.75	& 4.0767e-10	& 4.0767e-10	& 1.7168e-10\\
\hline 1.70	& 1.4566e-09	& 1.4566e-09	& 1.2755e-09\\
\hline 1.65	& 6.6352e-09	& 6.6352e-09	& 6.9281e-09\\
\hline 1.6	& 2.3116e-09	& 2.3116e-09	& 1.5905e-10\\
\hline 1.55	& 6.6989e-09	& 6.6989e-09	& 1.0847e-09\\
\hline 1.5	& 2.6565e-06	& 2.6565e-06	& 1.1093e-07\\
\hline 1.45	& 2.6856e-08	& 2.6856e-08	& 1.6230e-09\\
\hline 1.40	& 4.5396e-07	& 4.5396e-07	& 1.0716e-07\\
\hline 1.35	& 5.7719e-04	& 5.7719e-04	& 4.0767e-07\\
\hline 1.30	& 4.0564e-04	& 4.0564e-04	& 1.6920e-07\\
\hline 1.25	& 8.1159e-06	& 8.1159e-06	& 1.0746e-07\\
\hline 1.2	& 1.0879e-05	& 1.0879e-05	& 2.8224e-08\\
\hline 1.15	& 1.5509e-06	& 1.5509e-06	& 4.1217e-10\\
\hline 1.10	& 1.1366e-04	& 1.1366e-04	& 3.2026e-06\\
\hline 1.05	& 1.1904e-04    & 1.1904e-04	& 6.9962e-07\\
\hline 1	& 1.4001e-05	& 1.4001e-05	& 4.2669e-08\\
\hline 
\end{tabular}
\end{center}
\vspace{.1in}
\caption{Norm of distance between Nash solution and desired function when fractional Laplacian PDE is considered} \label{tab2}
\end{table}

\end{example}

\begin{example}\label{ex:3}
We consider a one-dimensional fractional space-time parabolic Nash equilibrium problem 
on the domain $\Omega = (-L, L)$ with $L=2$, discretized using $N=40$ interior nodes. 
The final time is $T=0.5$, with $M=80$ uniform time steps. 
The Caputo time derivative is taken of order $\gamma \in \{0.60, 0.70, 0.80, 0.90\}$, 
and the spatial fractional Laplacian has order $s \in \{0.25, 0.5, 0.75\}$. 

The state solution is initialized with $w_0(x) \equiv 0$ and subjected to a constant 
source term $f(x) \equiv 1$. The desired states are defined as characteristic functions on the subdomains $\omega_1$ and $\omega_2$
\[
w_{1}(x) =
\begin{cases}
1, & x \in \omega_1, \\
0, & \text{otherwise},
\end{cases}
\qquad
w_{2}(x) =
\begin{cases}
1, & x \in \omega_2, \\
0, & \text{otherwise},
\end{cases}
\]
where $\omega_1$ and $\omega_2$ denote the left and right quarter subdomains of $\Omega$, respectively. 
The control functions $u_1$ and $u_2$ are spatially distributed, time-independent, 
and initially taken as unit vectors supported in $\omega_1$ and $\omega_2$. 
The regularization parameters are $\mu_1 = \mu_2 = 10$.

The optimization is performed using a conjugate gradient (CG) algorithm: the forward 
state is computed using the L1 scheme for the left Caputo derivative, and the adjoint 
equation is solved by time reversal with the right Caputo derivative. Gradients are 
assembled from the adjoints, and the controls are iteratively updated.

Table~\ref{tab:fractional_errors_0} summarizes the $L^2$-errors for the forward state 
at final time for different combinations of $(\gamma, s)$.

\begin{table}[h!]
\centering
\begin{tabular}{|c|c|c|c|}
\hline 
$\gamma$ & $s$ & $\|\bar w - w_{1}\|_{L^2(\Omega)}$ & $\|\bar w - w_{2}\|_{L^2(\Omega)}$ \\ \hline
0.60 & 0.50 & 6.676e-01 & 6.676e-01 \\ \hline
0.60 & 1.00 & 6.191e-01 & 6.191e-01 \\ \hline
0.60 & 1.50 & 6.282e-01 & 6.282e-01 \\ \hline
0.70 & 0.50 & 6.608e-01 & 6.608e-01 \\ \hline
0.70 & 1.00 & 6.193e-01 & 6.193e-01 \\ \hline
0.70 & 1.50 & 6.282e-01 & 6.282e-01 \\ \hline
0.80 & 0.50 & 6.523e-01 & 6.523e-01 \\ \hline
0.80 & 1.00 & 6.196e-01 & 6.196e-01 \\ \hline
0.80 & 1.50 & 6.281e-01 & 6.281e-01 \\ \hline
0.90 & 0.50 & 6.420e-01 & 6.420e-01 \\ \hline
0.90 & 1.00 & 6.195e-01 & 6.195e-01 \\ \hline
0.90 & 1.50 & 6.279e-01 & 6.279e-01 \\ \hline
\end{tabular}
\caption{Forward state errors for different $(\gamma,s)$ at final time $T$.}
\label{tab:fractional_errors_0}
\end{table}

\end{example}

\begin{example}\label{ex4}
This example considers the same fractional Nash problem with modified desired targets $w_1 = 1$ on $\omega_1$ and $w_2 = -1$ on $\omega_2$, where $\omega_1$ and $\omega_2$ are the same as in Example 3. The spatial domain, discretization, time stepping, and Caputo and Laplacian orders remain as in the previous example.

Table~\ref{tab:fractional_errors} reports the errors of the computed Nash solutions for various combinations of $(\gamma, s)$. The errors are computed using the Euclidean norm $\|w - w_i\|_2$. 

\begin{table}[h!]
\centering
\begin{tabular}{|c|c|c|c|}
\hline 
$\gamma$ & $s$ & $\|\bar w - w_{1}\|_{L^2(\Omega)}$ & $\|\bar w - w_{2}\|_{L^2(\Omega)}$ \\ 
\hline
0.60 & 0.25 & 3.252 & 5.346 \\ \hline
0.60 & 0.50 & 3.013 & 4.776 \\\hline
0.60 & 0.75 & 2.857 & 4.239 \\\hline
0.70 & 0.25 & 3.173 & 5.237 \\\hline
0.70 & 0.50 & 2.983 & 4.734 \\\hline
0.70 & 0.75 & 2.853 & 4.235 \\\hline
0.80 & 0.25 & 3.086 & 5.109 \\\hline
0.80 & 0.50 & 2.945 & 4.678 \\\hline
0.80 & 0.75 & 2.844 & 4.225 \\\hline
0.90 & 0.25 & 2.995 & 4.965 \\\hline
0.90 & 0.50 & 2.900 & 4.605 \\\hline
0.90 & 0.75 & 2.831 & 4.205 \\\hline
\end{tabular}
\caption{Nash solution errors for different $\gamma$ and $s$.}
\label{tab:fractional_errors}
\end{table}

As observed, the errors are larger than in the previous example due to the opposite signs of the target functions in $\omega_1$ and $\omega_2$. Nonetheless, the results indicate partial convergence toward the desired Nash states, with slightly improved approximation as $\gamma$ and $s$ increase.
\end{example}

\section{Conclusion}
The Nash equilibrium for a non-cooperative differential game with two players (controllers) who each aim to minimize their own cost functional has been examined in this work. Space-time fractional partial differential equations that incorporate the fractional Laplacian in space and the Caputo time derivative govern the system. We have established the existence and uniqueness of the Nash equilibrium under appropriate conditions and shown that the equilibrium can be characterized as the solution to a bi-objective optimal control problem. Our numerical results of the experiment, which are compiled in Tables \ref{tab1}, \ref{tab2}, \ref{tab:fractional_errors_0}, and \ref{tab:fractional_errors}, support the theoretical conclusions and show how well the suggested computational  works to solve these optimal fractional control problems.

\section{Future Work}
In the future, we plan to extend this study to a fractional Nash–Stackelberg setting, where one player acts as a leader and the others as followers. This method is inspired by the hierarchical control strategies introduced by J.-L. Lions in \cite{lions1994remarks}, and we aim to explore its extension in the fractional PDE Model. As a natural extension, we are interested in exploring fractional PDE models within the  Stackelberg–Nash strategies \cite{diaz1994stackelberg}, where one leader interacts with multiple followers seeking a Nash equilibrium.

\section*{Declarations}

\textbf{Conflicts of Interest:} The authors declare that they have no conflicts of interest.

\end{document}